\documentclass[12pt]{article}

\usepackage{amsthm,amsmath,amssymb}
\usepackage{mathtools}
\usepackage{graphicx}
\usepackage[T1]{fontenc}

\usepackage[colorlinks=true,citecolor=black,linkcolor=black,urlcolor=blue]{hyperref}


\newcounter{parentnumber}
\theoremstyle{plain}
\newtheorem{theorem}{Theorem}

\theoremstyle{lemma}

\theoremstyle{proposition}
\newtheorem{proposition}{Proposition}

\theoremstyle{corollary}
\newtheorem{corollary}{Corollary}

\theoremstyle{definition}
\newtheorem{definition}{Definition}

\newtheorem{problem}[theorem]{Problem}

\newtheorem{solution}{Solution}

\theoremstyle{remark}

 \title{   A Simple Algorithm for a Computationally Hard Problem  }

\author{Ameneh Farhadian \thanks{The contents of this paper are taken from the author's Ph.D. Thesis, Department of Mathematical Sciences, Sharif University of Technology supervised by Prof. Mahmoodian.
} \\
\\
\small Sharif University of Technology\\[-0.8ex]
\small   Tehran, I. R. Iran\\
\small\tt  September 11, 2019\\ }

\date{ \small Mathematics Subject Classifications:  05C85 , 05C60}

\begin{document}

\maketitle

\begin{abstract}
 Graph isomorphism problem is a known hard problem. In this paper, a randomized algorithm is proposed for this problem which is  very simple and fast. It solves the graph isomorphism problem with running time $O(n^{2.373})$ for any pair of $n$-vertex graphs which are not strongly co-det.

\end{abstract}

\section {Introduction}

In this paper, any matrix  is a $n$-by-$n$ matrix.
$I$ and $J$  denote the identity matrix and all one matrix, respectively.  A \textit{diagonal matrix} is a matrix in which the entries outside the main diagonal are all zero. A \textit{scalar matrix} is a special kind of diagonal matrix. It is a diagonal matrix with equal-valued elements along the diagonal. A \textit{block diagonal} matrix is a square diagonal matrix in which the diagonal elements are square matrices of any size, and the off-diagonal elements are 0.  A \textit{random matrix} is a matrix-valued random variable, i.e. a matrix in which some or all elements are random variables. Determinant of  matrix $A$ is denoted by $\det(A)$.
A \textit{permutation matrix} $P$ is a matrix obtained by permuting the rows of an identity matrix according to some permutation of the numbers 1 to $n$.
Two $n$-by-$n$ matrices $A$ and $B$ are \textit{similar}, if there exists an invertible matrix $S$ such that $A=SBS^{-1}$.
Matrices $A$ and $B$ are \textit{permutation similar}, if there exists a permutation matrix $P$ such that $A=PBP^t$. Let $A$ be a square matrix. We define $\rm{diag}(A)$ as a square diagonal matrix  with the  main diagonal  of  matrix $A$. 

Suppose that two matrices are given. We want to decides whether they are permutation similar. In other words, whether there is a bilateral bijection between the rows (and columns) of the two matrices. To check this fact, one should consider each of the $n!$ bijections between rows of two matrices and check whether two matrices are the same by that permutation. If they are not permutation-similar, one would need to check all $n!$ bijections to realize this fact. But, even for relatively small values of $n$, the number $n!$ is unmanageably large.

The  graph isomorphism problem is a special case of this problem. Two graphs are isomorphic if and only if their adjacency matrices are permutation-similar.
The graph isomorphism problem is a known hard problem. Although there exist  efficient algorithms for this problem in practice, but no polynomial time algorithm is known for this problem.  The best existing algorithm for this problem has exponential run time, i.e. $2^{O( \sqrt{n \log {n} }) }$ \cite{babai1983canonical}. The most recent work on this problem \cite{babai2016graph} shows that there exists an algorithm with running time $ 2^{O((\log n)^{c})} $ for some fixed $ c>0$. But, its proof has not been fully peer-reviewed yet.

In this paper, we prove that there exists a suitable function when applied to matrices, the determinant of the obtained matrices determines that whether  two matrices are permuationa-similar or not. Therefore, checking the isomorphism of two graphs is  reduced to  computation of  determinant of adjacency matrices which can be computed in time $O(n^{2.373})$.
\section{The main idea}
To warm up, first we deal with a preliminary problem. The solution of this  problem illuminates  our approach to the main problem.
\begin{problem}\label{p1}
Suppose that two persons have two square matrices $A$ and $B$ as their secrets. We want to decide whether $ A=B$ or not. They tell us only the determinant of their matrices. Also, they do any function that we want on their matrices. But, they tell us just the determinant of the resulted matrix.  Can we decide whether $A=B$ or not?
\end{problem}
\begin{solution}
We choose an arbitrary matrix $ X$ (which is not a scalar matrix) and ask them to add $X$ to their matrices, i.e. $A$ and $B$. Clearly, matrices $A+X$ and $B+X$ are the same, if $A=B$. In opposite, if $A \neq B$, then $A+X$ and $B+X$ are two  different matrices. We show that $\mathbb {P}\lbrace \det(A+X)=\det(B+X) \rbrace=0$ for randomly chosen matrix $X$.
We define function $F=\det(A+X)-\det(B+X)$. Clearly, function F is a multivariate polynomial in terms of  elements of matrix $X$, say $x_{1,1}, x_{1,2}, \cdots, x_{n,n}$. The coefficients of this multivariate polynomial is a function of elements of matrices $A$ and $B$. If $A=B$, then $F\equiv 0$. If $A\neq B$, then $F $ is a non-zero multivariate polynomial. According to  Schwartz-Zippel lemma \cite{schwartz1980fast,zippel1979probabilistic} the probability that the  randomly chosen $x_{1,1}, \cdots ,x_{n,n}$ is a root of $F$ is zero. Therefore, if $A \neq B$ we have  $\mathbb {P}\lbrace \det(A+X)=\det(B+X) \rbrace=0$ for randomly chosen matrix $X$. \\
According to the above explanation, for a randomly chosen matrix $X$, if $\det(A+X) \neq \det(B+X)$, then we can  certainly deduce  $A\neq B$. If $\det(A+X) = \det(B+X)$, then we deduce that $A=B$ with probability 1. Because, the probability that two randomly chosen matrices share the same eigenvalues is zero.  The diagram of this algorithm is depicted in Fig. \ref{p1}.\end{solution}

\begin{figure}[ht]
\centerline{\includegraphics[width=7.8cm]{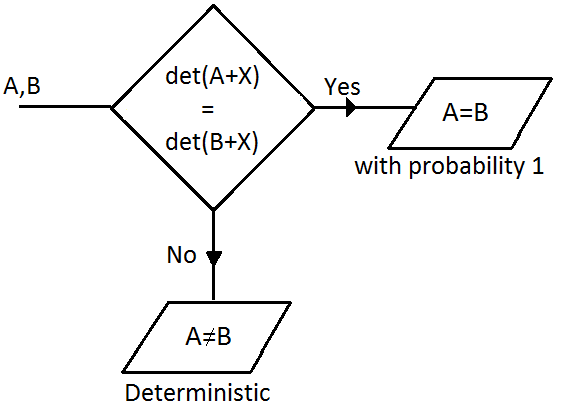}}
\caption{\label{p1}\small Deciding whether two matrices $A$ and $B$ are equal. $X$ is a random matrix. }
\end{figure}
\begin{problem}\label{p2}
 Two matrices $A$ and $B$ are given. We want to decide whether $A$ and $B$ are permutation-similar, i.e. is there a permutation matrix $P$ such that $A=PBP^{t}$? Similar to  Problem \ref{p1}, we are restricted to compare the eigenvalues of the matrices.
\end{problem}

We want  to solve  Problem \ref{p2} similar to Problem \ref{p1}. 
Finding a solution for Problem 2 is important. Because, deciding whether two matrices are permutation-similar is a hard problem, that is in worst cases all possible permutations should be checked which needs exponential time.
 Additionally, the graph isomorphism problem, which is a known hard problem, is a special case of this problem. Therefore, finding a fast and easy algorithm to solve this problem  is valuable. In the  graph isomorphism problem, the matrices $A$ and $B$ are  zero-one symmetric matrices. 
As it is depicted in Fig. \ref{p2}, we need   a suitable function $f()$.
\begin{figure}[ht]
\centerline{\includegraphics[width=7.8cm]{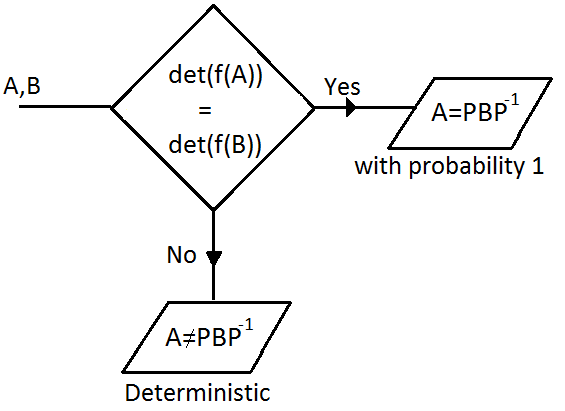}}
\caption{\label{p2}\small A simple algorithm to decide whether two matrices $A$ and $B$ are permutation-similar}
\end{figure}
\section{Looking for a suitable function}
Here, we look for a suitable function $f$ for the algorithm depicted in Fig \ref{p2}.
Two matrices $A$ and $B$ are given. We should decide whether two matrices $A$ and $B$ are permutation-similar. We require  a function $f()$ acting on matrices $A$ and $B$ such that $\det(f(A))=\det(f(B))$ holds with probability 0, if two matrices $A$  and $B$ are not permutation-similar.
In addition, function $f()$ should be such that
$\det(f(A))=\det(f(B))$ if $A$ and $B$ are permutation-similar.
 Therefore, the desired function $f()$ should satisfy  the following two properties,
\begin{itemize}
\item Property 1:
 $\det(f(A))=\det(f(B))$, if there exists a permutation matrix $P$ such that $A=PBP^t$,\\
\item Property 2:   $ \mathbb{P} \lbrace \det(f(A)) = \det(f(B)) \rbrace =0$, if  there is not any permutation matrix such that $A= PBP^{t}$.
\end{itemize}

  Is it possible to find such function $f()$?  Clearly, a polynomial function satisfy in the first property, but it does not satisfy  the second one. The function of adding a random matrix, i.e. $f(A)=A+X$, similar to what we have done for  Problem \ref{p1}, satisfies the second property, but it does not satisfy the first one.
Now, we look for a function which satisfies both of the above mentioned properties.
Similar to the solution of Problem 1, we  want to add a random matrix $X$ to matrices $A$ and $B$ such that the equality of their determinant is preserved if they are permutation similar. In opposite, if they are not permutation similar, we have  $\mathbb{P} \lbrace \det(A+X)=\det(B+X)\rbrace=0$.

In  Problem \ref{p1}, it can be easily checked that  if $X$ is a random diagonal matrix, then, again,  the result remains the same. Here, we  suggest a random diagonal matrix to be added to matrices $A$ and $B$.
%
%
 Let $c$ be a real number and  $J$ be all one matrix. Clearly, any entry of matrix $(A+cJ)^k $ for integer $k>2$ is a function of all entries of $A$. If $A$ is a block diagonal  matrix, then it is possible that some entries of $A^k$ does not depend on all entries of $A$.
 Thus, $cJ$ is added to $A$ to be sure that $A+cJ$ is not a block diagonal matrix and any element of   $(A+cJ)^k $ is a function of all entries of $A$.
\begin{proposition}\label{cJ}
 Two square matrices $A$ and $B$ are permutation similar, if and only if $A+cJ$ and $B+cJ$  are permutation similar for any arbitrary  $c$.
\end{proposition}

Therefore, without loose of generality, we can assume that matrices $A$ and $B$ are not block diagonal matrices. It means that any entry of $A^k$ is a function of  all entries of $A$ for any integer $k>2$.

Assuming $q(x)=\sum_{i=1}^nc_ix^i$, matrix ${\rm{diag}}(q(A))$ is a diagonal matrix in which any diagonal entry is a multivariate polynomial in terms of all entries of $A$ with randomized coefficients $c_i$. We show that matrix ${\rm{diag}}(q(A))$ can play the same role of  matrix $X$ in the solution of Problem \ref{p1}, except some special cases.

\begin{definition}\label{sc}
Let $A$ and $B$ be two matrices which are not permutation-similar.  We say $A$ and $B$ are \textit{co-det}, if $\det(A)=\det(B)$. We say $A$ and $B$ are \textit{strongly co-det}, if for any arbitrary finite order polynomial $q(x)=\sum c_ix^i$ and real $c$, we have
$$\det(q(A+cJ)-{\rm{diag}}(q(A+cJ))=\det(q(B+cJ)-{\rm{diag}}(q(B+cJ)))$$
\end{definition}
We know that  any matrix satisfies in its characteristic polynomial and the degree of  characteristic polynomial of an $n$-by-$n$  matrix is  $n$. Thus, any polynomial in terms of an $n$-by-$n$ matrix can be reduced to a polynomial with degree at most $n$. Therefore, it  is sufficient to suppose that the degree of polynomial $q()$ in the above definition  is at most $n$ where $n$ is the size of matrices $A$ and $B$. \\
 According to the above definition, the strongly co-det matrices are very special matrices. Generally, $A^k$ and $\rm{diag}(A^k)$ are two different matrices with different eigenspace. Thus, for any integer $k$, $A^k-\rm{diag}(A^k)$ is a matrix whose eigenspace is neither  the eigenspace of $A^k$ nor  the eigenspace of $\rm{diag}(A^k)$ in general. Therefore, the strongly co-det matrices which share the same determinant for$A^k-\rm{diag}(A^k)$ for any integer $k$, have  special structure.

\begin{theorem}\label{diagf} Suppose that  $c, c_1, \cdots, c_n \in \mathbb{R} $ are selected at random. We define function $f()$ acting on $n$-by-$n$ matrices, by
$$f(A)=\sum_{i=1}^{n}c_i(A+cJ)^i-{\rm{diag}}(\sum_{i=1}^{n}c_i(A+cJ)^i)$$
 Function $f$ satisfies both properties 1 and 2, except the case that two matrices $A$ and $B$  are strongly co-det.
\end{theorem}
\begin{proof}
First we show that function $f()$  satisfies property1. That is, given two matrices $A$ and $B$, if there exists permutation matrix $P$ such that $A=PBP^{t}$, then  $\det(f(A))=\det(f(B))$. We have
$$f(A)= \sum_{i=1}^{n}c_i(A+cJ)^i-{\rm{diag}}(\sum_{i=1}^{n}c_i(A+cJ)^i)$$
 If  there exists a permutation matrix $P$ such that $A=PBP^{t}$, then
$$f(A)=\sum_{i=1}^{n}c_i (PBP^{t}+cJ)^i-{\rm{diag}} (\sum_{i=1}^{n}c_i (PBP^{t}+cJ)^i)$$
Since $J=PP^tJ=PJP^t$ and $ (PXP^t)^i=P X^iP^t$,
$$f(A)=P(\sum_{i=1}^{n}c_i (B+cJ)^i)P^t-{\rm{diag}} (P(\sum_{i=1}^{n}c_i (B+cJ)^i)P^t)$$
 In addition, for any permutation matrix $P$ and matrix $A$, we have
 ${\rm{diag}}(PAP^{t})=P{\rm{diag}}(A)P^t$. Thus,
$$f(A)=P(\sum_{i=1}^{n}c_i(B+cJ)^i)P^t-P{\rm{diag}}(\sum_{i=1}^{n}c_i(B+cJ)^i)P^t$$
$$= P( \sum_{i=1}^n(B+cJ)-{\rm{diag}}(\sum_{i=1}^{n}c_i(B+cJ)^i))P^t=Pf(B)P^t$$
Since $\det(PXP^t)=\det(PP^tX)=\det(X)$, we conclude that $\det(f(A))=\det(f(B))$. It means that  function $f$  satisfies property 1.

Now, we show the satisfaction of property 2. Given two matrices $A$ and $B$ which are not strongly-co-det, we show that for randomly chosen values of $c_1, \cdots, c_n$ and $c$, $ \mathbb{P} \lbrace \det(f(A)) = \det(f(B)) \rbrace =0$, if there is no permutation matrix $P$ such that $A=PBP^t$. Clearly, the probability space is $\mathbb{R}^{n+1}$.
First, we define function $G$ as
$$ G(c,c_1, \cdots, c_n)=\det(f(A))-\det(f(B))$$
where $c_1,\cdots , c_n$ and $c$ are coefficients which have been used for definition of $f()$.
 Clearly, any element of matrix $f(A)$ is a summation of $c_i$'s. Consequently, $\det(f(A))$ is a multivariate polynomial in terms  of  $c, c_1, \cdots $ and $ c_n$.  We have assumed two matrices $A$ and $B$ are not strongly-co-det. Thus,  there exist a  polynomial $q_0()$ and a real number $c_0$ such that $\det(q_0(A+c_0J)-{\rm{diag}}(q_0(A+c_0J))) \neq \det(q_0(B+c_0J)-{\rm{diag}}(q_0(B+c_0J)))$. Thus, $G(c_0,c'_1, \cdots, c'_n) \neq 0$ where $c'_i$ are the coefficients of the polynomial $q_0()$. It means that $G$  is a non-zero multivariate polynomial.  The Schwartz-Zippel lemma \cite{schwartz1980fast,zippel1979probabilistic} states that for a finite order non-zero multivariate polynomial $F(x_1,\cdots,x_k)$, the probability that $F(r_1, \cdots,r_k)=0$ holds for  randomly chosen $r_1, \cdots,r_k$  is zero. Therefore, the probability that $ G(c, c_1, \cdots, c_n)=0 $  holds for randomly chosen values of $c,c_1, \cdots, c_n$ is zero. Thus, for randomly chosen coefficients $c , c_1 , \cdots$ and $c_n$ in the definition of function $f()$, we have
$$\mathbb {P} \lbrace \det(f(A))= \det(f(B)) \rbrace =0$$
, if $A$ and $B$ are not strongly-co-det and $A \neq PBP^t$.$ \square$

\end{proof}
%
\begin{corollary}
The computational complexity of checking that two matrices $A_1$ and $A_2$ are permutation-similar is equal to the computational complexity of checking  that $\det(f(A_1))=\det(f(A_2))$ holds, provided that $A_1$ and $A_2$ are not strongly co-det.
\end{corollary}
The problem of computing the determinant of a square matrix is a known problem. According to Theorem 6.6 of ref \cite{aho1974design}, the determinant of an $n$-by-$n$ matrix can be computed in $O(n^{2.373})$ steps.
\begin{corollary}
The problem of   graph isomorphism  is solvable in time  $O(n^{2.373})$ for any two $n$-vertex graphs  that their adjacency matrices are not strongly co-det.
\end{corollary}
According to the above result, the computational complexity of the graph isomorphism is an open problem just for graphs with strongly co-det adjacency matrices.
Clearly, such graphs   have a special symmetric structure due to the definition. By manipulating their adjacency matrices, they will be not strongly co-det and the above simple algorithm can be applied for them, for instance by addition of a non-zero value to one element of the diagonal.  More study about strongly co-det matrices  is suggested for the future.

\bibliographystyle{plain}




\end{document}